\documentclass[12pt, a4paper]{article}
\usepackage{graphicx}
\usepackage{amsfonts}
\usepackage{amsmath}
\usepackage{amssymb}
\usepackage{amsthm}
\usepackage{mathtools}

\usepackage[ansinew]{inputenc}
\usepackage{latexsym}
\usepackage{tabularx}
\usepackage{multirow}

\usepackage[active]{srcltx}
\usepackage[usenames]{color}
\allowdisplaybreaks
\newtheorem{thm}{Theorem}
\newtheorem{cor}{Corollary}
\newtheorem{lem}{Lemma}

\newtheorem{rem}{Remark}

\newtheorem{examp}{Example}
\newtheorem{algorithm}{Algorithm}

\newcommand{\abs}[1]{\left\vert#1\right\vert}

\newcommand{\satop}[2]{\stackrel{\scriptstyle{#1}}{\scriptstyle{#2}}}

\newcommand{\bsDelta}{\boldsymbol{\Delta}}
\newcommand{\bsOmega}{\boldsymbol{\Omega}}

\newcommand{\bsgamma}{\boldsymbol{\gamma}}

\newcommand{\bsx}{\boldsymbol{x}}
\newcommand{\bsz}{\boldsymbol{z}}

\newcommand{\bsh}{\boldsymbol{h}}
\newcommand{\bsg}{\boldsymbol{g}}
\newcommand{\bst}{\boldsymbol{t}}

\newcommand{\bsw}{\boldsymbol{w}}

\newcommand{\bsy}{\boldsymbol{y}}

\newcommand{\D}{{\cal D}}
\newcommand{\cH}{{\cal H}}
\newcommand{\cP}{{\cal P}}

\newcommand{\wal}{{\rm wal}}
\newcommand{\sob}{{\rm sob}}

\newcommand{\icomp}{\mathtt{i}}
\newcommand{\bszero}{\boldsymbol{0}}

\newcommand{\rd}{\,\mathrm{d}}
\newcommand{\Field}{\mathbb{F}}

\newcommand{\NN}{\mathbb{N}}
\newcommand{\ZZ}{\mathbb{Z}}

\newcommand{\FF}{\mathbb{F}}

\newcommand{\cZ}{{\cal Z}}
\newcommand{\cU}{{\cal U}}

\newcommand{\uu}{\mathfrak{u}}
\newcommand{\vv}{\mathfrak{v}}

\makeatletter
\newcommand{\rdots}{\mathinner{\mkern1mu\lower-1\p@\vbox{\kern7\p@\hbox{.}}
\mkern2mu \raise4\p@\hbox{.}\mkern2mu\raise7\p@\hbox{.}\mkern1mu}}
\makeatother

\date{}

\begin{document}

\title{A reduced fast component-by-component construction of lattice points for integration in weighted spaces with fast decreasing weights}

\author{Josef Dick\thanks{J. Dick is supported by a QEII Fellowship of the Australian Research Council.}, 
Peter Kritzer\thanks{P. Kritzer is supported by the Austrian Science Fund (FWF): 
Project P23389-N18 and Project F5506-N26, which is a part of the Special Research Program "Quasi-Monte Carlo Methods: Theory and Applications".}, 
Gunther Leobacher\thanks{G. Leobacher is supported by the Austrian Science Fund (FWF): Project F5508-N26, 
which is a part of the Special Research Program "Quasi-Monte Carlo Methods: Theory and Applications".}, 
Friedrich Pillichshammer\thanks{F. Pillichshammer is supported by the Austrian Science Fund (FWF): Project F5509-N26, 
which is a part of the Special Research Program "Quasi-Monte Carlo Methods: Theory and Applications".}}

\maketitle

\begin{abstract}
Lattice rules and polynomial lattice rules are quadrature rules for approximating integrals over the $s$-dimensional unit cube. 
Since no explicit constructions of such quadrature methods are known for dimensions $s > 2$, one usually has to resort to computer search algorithms. 
The fast component-by-component approach is a useful algorithm for finding suitable quadrature rules.

We present a modification of the fast component-by-component algorithm which yields savings of the construction cost for (polynomial) lattice rules in weighted function spaces. The idea is to reduce the size of the search space for coordinates which are associated with small weights and are therefore of less importance to the overall error compared to coordinates associated with large weights. We analyze tractability conditions of the resulting QMC rules. Numerical results demonstrate the effectiveness of our method.
\end{abstract}
 
 


\section{Introduction}

In this paper we study the construction of quasi-Monte Carlo (QMC) rules
\begin{equation}\label{eq_qmc_rule}
\frac{1}{N} \sum_{n=0}^{N-1} f(\bsx_n) \approx \int_{[0,1]^s} f(\bsx) \rd \bsx,
\end{equation}
which are used for the approximation of $s$-dimensional integrals over the unit cube $[0,1]^s$. 
We consider two types of quadrature point sets $\cP = \{\bsx_0, \bsx_1, \ldots, \bsx_{N-1}\}$, namely, 
lattice point sets \cite{DKS13, Nie92, SJ94} and polynomial lattice point sets \cite{L04, Nie92, P10}. 
For a natural number $N \in \mathbb{N}$ and a vector $\bsz \in \{1, 2, \ldots, N-1\}^s$, a lattice point set is of the form
\begin{equation*}
\left\{ \frac{k}{N} \bsz \right\}\ \ \ \mbox{ for }\ \ k=0,1,\ldots,N-1.
\end{equation*}
Here, for real numbers $x \ge 0$ we write $\{x\} = x - \lfloor x \rfloor$ for the fractional part of $x$. 
For vectors $\bsx$ we apply $\{\cdot \}$ component-wise. Polynomial lattice point sets are similar, 
one only replaces the arithmetic over the real numbers by arithmetic of polynomials over finite fields. 
More details about polynomial lattice point sets are given in Section~\ref{secnotation}.

In order to analyze the quality of a given point set $\cP$ with respect to their performance in a QMC rule \eqref{eq_qmc_rule}, one usually considers the worst-case integration error in the unit ball of a Banach space $(\cH, \|\cdot\|)$ given by
\begin{equation*}
e_{N,s}(\cH, \cP) = \sup_{\satop{f \in \cH}{\|f\| \le 1}} \left|\int_{[0,1]^s} f(\bsx) \rd \bsx - \frac{1}{|\cP|} \sum_{\bsx \in \cP} f(\bsx) \right|.
\end{equation*}
Here we restrict ourselves to certain weighted reproducing kernel Hilbert spaces $H$. For background on reproducing kernel Hilbert spaces see \cite{A50} and for reproducing kernel Hilbert spaces in the context of numerical integration see \cite{H98, SW98}. In case of lattice rules, we consider the so-called weighted Korobov space (Section~\ref{sec_Korobov_space}) and for polynomial lattice rules we consider weighted Walsh spaces (Section~\ref{secnotation}).

The paper \cite{SW98} introduced weighted reproducing kernel Hilbert spaces. 
The weights are a sequence of non-negative real numbers $(\gamma_\uu)_{\uu \subseteq [s]}$, 
where $[s] = \{1,2, \ldots, s\}$, which model 
the importance of the projection of the integrand $f$ onto the variables $x_j$ for $j \in \uu$. 
A small weight $\gamma_\uu$ means that the projection onto the variables in $\uu$ contributes little to the integration problem. 
A simple choice of weights are so-called product weights $(\gamma_j)_{j \in \mathbb{N}}$, where $\gamma_\uu = \prod_{j\in \uu} \gamma_j$. 
In this case, the weight $\gamma_j$ is associated with the variable $x_j$.

We introduce the concept of tractability \cite{NW10}. Let $e(N,s)$ be the $N$th minimal QMC worst-case error
\begin{equation*}
e(N,s) = \inf_{\cP} e(\cH, \cP),
\end{equation*}
where the infimum is extended over all $N$-element point sets $\cP$ in $[0,1]^s$. 
We also define the initial error $e(0,s)$ as the integration error when approximating the integral by $0$, that is,
\begin{equation*}
e(0,s) = \sup_{\satop{f \in \cH}{\|f\| \le 1}} \left| \int_{[0,1]^s} f(\bsx) \rd \bsx \right|.
\end{equation*}
This is used as a reference value.

We are interested in the dependence of the worst-case error on the dimension $s$. We consider the QMC information complexity, which is defined by
\begin{equation*}
N_{\min}(\varepsilon, s) = \min\{ N \in \mathbb{N}: e(N,s) \le \varepsilon e(0,s)\}.
\end{equation*}
This means that $N_{\min}(\varepsilon, s)$ is the minimal number of points which are required to reduce the initial error by a factor of $\varepsilon$.

We can now define the following notions of tractability. We say that the integration problem in $H$ is 
\begin{enumerate}
\item weakly QMC tractable, if
\begin{equation*}
\lim_{s + \varepsilon^{-1} \to \infty} \frac{\log N_{\min}(\varepsilon, s)}{s+\varepsilon^{-1}}  = 0;
\end{equation*}

\item polynomially QMC-tractable, if there exist non-negative numbers $c$, $p$ and $q$ such that
\begin{equation}\label{ineq_Nmin}
N_{\min}(\varepsilon, s) \le c s^q \varepsilon^{-p};
\end{equation}

\item strongly polynomially QMC-tractable, if \eqref{ineq_Nmin} holds with $q = 0$. 
We call the infimum over all $p$ such that \eqref{ineq_Nmin} holds the $\varepsilon$-exponent of (strong) polynomial tractability.
\end{enumerate}

It is known that, in order to achieve strong polynomial tractability of the integration problem in the weighted Korobov space with product weights $(\gamma_j)_{j \in \mathbb{N}}$, it is necessary and sufficient (see \cite{SW01}) to have $$\sum_{j=1}^\infty \gamma_j < \infty.$$ It is also well known, see \cite{K03}, that if $\sum_{j=1}^\infty \gamma_j^{1/\tau} < \infty$ for some $1 \le \tau < \alpha$, then one can set the $\varepsilon$-exponent to $\tau/2$. In \cite{K03} for $N$ prime and in \cite{D04} for arbitrary $N$, it was shown that suitable lattice rules can be constructed component-by-component. The construction cost of this algorithm was reduced to $O(sN \log N)$ operations by \cite{NC06a, NC06b}. Assume now that
\begin{equation}\label{sum_weights_tau}
\sum_{j=1}^\infty \gamma_j^{1/\tau} < \infty
\end{equation}
for some $\tau > \alpha$. Then no further advantage is obtained from \cite{D04, K03, NC06a, NC06b}, since one can get strong polynomial tractability with the optimal $\varepsilon$-exponent and the construction cost of the lattice rule is independent of the choice of weights. The aim of this paper is to take advantage of a situation where \eqref{sum_weights_tau} holds for $\tau > \alpha$, by showing that in this case one can reduce the construction cost of the lattice rule while still achieving strong polynomial tractability with the optimal $\varepsilon$-exponent.

The approach is the following. Let $b$ be a fixed prime number and let $N = b^m$ for $m \in \mathbb{N}$. We reduce the construction cost of the $j$th component of the lattice rule by reducing the size of the search space by a factor of $b^{w_j}$ for some integer $w_j \ge 0$. That is, instead of choosing the $j$th component of the generating vector from the set $\{z \in \{1,\ldots,N-1\}\ : \ \gcd(z, b) = 1\}$ we choose it from the set $\{z b^{w_j}: z \in \{1,\ldots ,N b^{-w_j}\} \mbox{ and } \gcd(z, b) = 1\}$ if $w_j < m$. The latter set is of size $N b^{-w_j} (b-1)/b$. If $w_j \ge m$ we set the $j$th component to $0$.  This reduction in the size of the search space reduces the construction cost of the fast component-by-component construction. Assume that the weights $\gamma_j$ are ordered such that $\gamma_1 \ge \gamma_2 \ge \gamma_3 \ge \cdots$. Then we can also order $w_j$ such that $0 \le w_1 \le w_2 \le w_3 \le \cdots$. Let $s^\ast$ be the smallest $j$ such that $w_{j} \ge m$. We show that the reduced fast component-by-component construction finds a lattice rule which achieves strong polynomial tractability if
\begin{equation*}
\sum_{j=1}^\infty \gamma_j^{1/\tau} b^{w_j} < \infty,
\end{equation*}
with a construction cost of $$O\left(N \log N + \min\{s, s^\ast\} N + \sum_{d=1}^{\min\{s, s^\ast\}} (m-w_d) N b^{-w_d} \right).$$  
Since the construction cost 
is, with respect to the dimension, 
limited by $s^\ast$, this means that if $w_j \to \infty$ as $j \to \infty$, we can set $s= \infty$. 
We present analogous results for weak tractability, polynomial tractability, and polynomial lattice rules.

The structure of this paper is as follows. In the next section we give background on weighted Korobov spaces with general weights. 
In Section~\ref{seccbclpspp} we present a reduced fast CBC construction of lattice rules achieving tractability 
in Korobov spaces, and in Section~\ref{sec:fast-mod-cbc} we discuss a way to efficiently implement the algorithm for the case of product weights. 
In Section~\ref{secnotation}, we state the results for Walsh spaces and polynomial lattice point sets. Finally, in the Appendix of the paper we demonstrate 
the proof of Theorem~\ref{thmcbclpspp}.

\section{The weighted Korobov space with general weights}\label{sec_Korobov_space}

We consider a weighted Korobov space with general weights as studied in \cite{DSWW06,NW10}. Before we do so we need to introduce some notation. Let $\ZZ$ be the set of integers and let $\ZZ_{\ast}=\ZZ \setminus\{0\}$. Furthermore, $\NN$ denotes the set of positive integers. For $s \in \NN$ we write $[s]=\{1,2,\ldots,s\}$. For a vector $\bsz=(z_1,\ldots,z_s)\in [0,1]^s$ and for $\uu \subseteq [s]$ we write $\bsz_\uu=(z_j)_{j \in \uu} \in [0,1]^{|\uu|}$ and $(\bsz_{\uu},\bszero)\in [0,1]^s$ for the vector $(y_1,\ldots,y_s)$ with $y_j=z_j$ if $j \in \uu$ and $y_j=0$ if $j \not\in \uu$. 

The importance of the different components or groups of components of the functions from the Korobov space to be defined will be specified with a sequence of positive weights $\bsgamma=(\gamma_{\uu})_{\uu \subseteq [s]}$, where we may assume that $\gamma_{\emptyset}=1$. The smoothness will be described with a parameter $\alpha>1$. 

The weighted Korobov space $\cH(K_{s,\alpha,\bsgamma})$ is a reproducing kernel Hilbert space with kernel function of the form 
\begin{align*}
K_{s,\alpha,\bsgamma}(\bsx,\bsy) & = 1+\sum_{\emptyset \not=\uu \subseteq [s]} \gamma_{\uu} \prod_{j \in \uu}\left(\sum_{h \in \ZZ_{\ast}} \frac{\exp(2 \pi \icomp h(x_j-y_j))}{|h|^{\alpha}}\right)\\
&= 1+ \sum_{\emptyset \not=\uu \subseteq [s]} \gamma_{\uu} \sum_{\bsh_{\uu}\in \ZZ_{\ast}^{|\uu|}} \frac{\exp(2 \pi \icomp \bsh_{\uu}\cdot (\bsx_{\uu}-\bsy_{\uu}))}{\prod_{j \in \uu}|h_j|^{\alpha}}
\end{align*}
and inner product
$$\langle f,g\rangle_{K_{s,\alpha,\bsgamma}}=\sum_{\uu \subseteq [s]} \gamma_{\uu}^{-1} \sum_{\bsh_{\uu}\in \ZZ_{\ast}^{|\uu|}} \left(\prod_{j \in \uu}|h_j|^{\alpha}\right) \widehat{f}((\bsh_{\uu},\bszero)) \overline{\widehat{g}((\bsh_{\uu},\bszero))},$$ where $\widehat{f}(\bsh)=\int_{[0,1]^s} f(\bst) \exp(-2 \pi \icomp \bsh \cdot \bst)\rd \bst$ is the $\bsh$th Fourier coefficient of $f$.

For $h \in \ZZ_{\ast}$, let $\rho_{\alpha}(h)=|h|^{-\alpha}$. For $\bsh=(h_1,\ldots,h_s) \in \ZZ_{\ast}^s$ define $\rho_{\alpha}(\bsh)=\prod_{j=1}^s \rho_{\alpha}(h_j)$.

It is known (see, for example, \cite{DSWW06}) that the squared worst-case error of a lattice rule generated by a lattice point $\bsz \in \ZZ^s$ in the weighted Korobov space $\cH(K_{s,\alpha,\bsgamma})$ is given by 
\begin{equation}\label{eqerrorexprlps}
e_{N,s,\alpha,\bsgamma}^2 (\bsz)=\sum_{\emptyset\neq\uu\subseteq [s]}\gamma_\uu 
\sum_{\bsh_{\uu}\in\D_\uu}\rho_\alpha (\bsh_\uu),
\end{equation}
where $$\D_\uu:=\left\{\bsh_\uu\in\ZZ_{\ast}^{\abs{\uu}}\ : \ \bsh_\uu\cdot\bsz_\uu\equiv 0\ (N) \right\}.$$
In the case of product weights, i.e., $\bsgamma_{\uu}=\prod_{j \in \uu}\gamma_j$, where $\gamma_j=\gamma_{\{j\}}$, this reduces to
\begin{equation}\label{eqerrorexprlpsprod}
e_{N,s,\alpha,\bsgamma}^2 (\bsz)=\sum_{\bsh \in \D} \rho'_{\alpha,\bsgamma}(\bsh),
\end{equation}
where $$\D:=\left\{\bsh\in\ZZ^s\setminus\{\bszero\}:\bsh\cdot\bsz\equiv 0\ (N) \right\},$$ where $\rho'_{\alpha,\gamma}(0)=1$ and $\rho'_{\alpha,\gamma}(h)=\gamma |h|^{-\alpha}$ for $h \in \ZZ_{\ast}$. For $\bsh=(h_1,\ldots,h_s)\in \ZZ^s$,  $\rho'_{\alpha,\bsgamma}(\bsh)=\prod_{j=1}^s \rho'_{\alpha,\gamma_j}(h_j)$.

\begin{rem} \label{re1} \rm
Consider a tensor product Sobolev space $\cH_{s,\bsgamma}^{\sob}$ of absolutely continuous functions whose mixed partial derivatives of order $1$ in each variable are square integrable. The norm in the unanchored weighted Sobolev space $\cH_{s,\bsgamma}^{\sob}$ (see \cite{H98}) is given by
\begin{equation*}
\| f\|_{\cH_{s,\bsgamma}^{\sob}} = \left(\sum_{\uu \subseteq [s]} \prod_{j\in \uu} \gamma_{\uu}^{-1} \int_{[0,1]^{|\uu|}} \left(\int_{[0,1]^{s-|\uu|}} \frac{\partial^{|\uu|}}{\partial \bsx_{\uu}} f(\bsx) \rd \bsx_{[s]\setminus \uu} \right)^2 \rd \bsx_{\uu}\right)^{1/2},
\end{equation*}
where $\partial^{|\uu|}f/\partial \bsx_{\uu}$ denotes the mixed partial derivative with respect to all variables $j \in \uu$. As pointed out in \cite[Section~5]{DKS13}, the root mean-square worst-case error $\widehat{e}_{N,s,\bsgamma}$ for QMC integration in $\cH_{s,\bsgamma}^{\sob}$ using randomly shifted lattice rules $(1/N)\sum_{k=0}^{N-1}f\left(\left\{ \frac{k}{N} \bsz+\bsDelta \right\} \right)$, i.e.,  $$\widehat{e}_{N,s,\bsgamma}(\bsz)=\left(\int_{[0,1)^s} e_{N,s,\bsgamma}^2(\bsz,\bsDelta) \rd \bsDelta\right)^{1/2},$$ where $e_{N,s,\bsgamma}(\bsz,\bsDelta)$ is the worst-case error for QMC integration in $\cH_{s,\bsgamma}^{\sob}$ using a randomly shifted integration lattice, is more or less the same as the worst-case error $e_{N,s,2,\bsgamma}$ in the weighted Korobov space $\cH(K_{s,2,\bsgamma})$ using the unshifted version of the lattice rules. In fact, we have 
\begin{equation} \label{eq:wceeqwce}
\widehat{e}_{N,s, 2 \pi^2 \bsgamma}(\bsz)=e_{N,s,2,\bsgamma}(\bsz),
\end{equation}
where $2 \pi^2 \bsgamma$ denotes the weights $( (2 \pi^2)^{|\uu|} \gamma_{\uu})_{\emptyset \not=\uu \subseteq [s]}$. For a connection to the so-called anchored Sobolev space see for example \cite[Section~4]{HW00} and \cite[Section~3]{SKJ02}. 

Alternatively, the random shift can be replaced by the tent transform $\phi(x) = 1 - |1-2x|$ in each variable. For a vector $\bsx \in [0,1]^s$ let $\phi(\bsx)$ be defined component-wise. Let $\widetilde{e}_{N,s, \bsgamma}(\bsz)$ be the worst-case error using the QMC rule $(1/N)\sum_{k=0}^{N-1}f\left(\phi\left(\left\{ \frac{k}{N} \bsz \right\} \right) \right)$. Then it is known from \cite[Lemma~1 and lines 11-13 on page 277]{DNP14} that
\begin{equation}\label{eq_wce_tent}
\widetilde{e}_{N,s, \pi^2 \bsgamma}(\bsz) = e_{N,s,2,\bsgamma}(\bsz),
\end{equation} 
where $\pi^2 \bsgamma = (\pi^{2|\uu|} \gamma_{\uu})_{\emptyset \neq \uu \subseteq [s]}$.

Thus, the results that will be shown in the following are
valid for the root mean-square worst-case error and the worst-case error using tent-transformed lattice rules for numerical integration in
the Sobolev space as well as for the worst-case error for numerical
integration in the Korobov space. Hence it suffices to state them only for $e_{N,s,\alpha,\bsgamma}$. Equation~(\ref{eq:wceeqwce}) can be used
to obtain results also for $\widehat{e}_{N,s,\bsgamma}$ and Equation~\eqref{eq_wce_tent} can be used to obtain results for $\widetilde{e}_{N,s,\bsgamma}$.

There is also a connection between the worst-case errors for numerical integration using polynomial lattice rules in the Walsh space and the anchored \cite[Section~5]{DKPS05} and unanchored Sobolev space \cite[Section~6]{DP05}. 
\end{rem}

\section{The reduced fast CBC construction}\label{seccbclpspp}

In this section we assume that $N$ is a prime power of the form $N=b^m$, where $b$ is a prime number and $m\in\NN$. Furthermore, let $w_1,\ldots,w_s\in\NN_0$ with $w_1\le w_2\le\cdots \le w_s$.  (The most interesting case is where $w_1 = 0$, since otherwise each point is just counted $b^{w_1}$ times.)

In what follows, for $w_j<m$, put $$\cZ_{N,w_j}:=\left\{z\in\{1,2,\ldots,b^{m-w_j}-1\}: \gcd(z,N)=1\right\},$$ and for $w_j\ge m$, put $$\cZ_{N,w_j}:=\{1\}.$$ Note that $$|\cZ_{N,w_j}|=\left\{
\begin{array}{ll}
b^{m-w_j-1}(b-1) & \mbox{ if } w_j<m,\\
1 & \mbox{ if } w_j \ge m.                                
\end{array}\right.$$
Furthermore, let $Y_j:=b^{w_j}$ for $j \in [s]$.

We propose the following reduced fast CBC construction algorithm for generating vectors $\bsz$. 
As $\alpha$ and $\bsgamma$ are fixed we suppress their influence on the worst-case error in the following and write simply 
$e_{N,s}(\bsz)$ instead of $e_{N,s,\alpha,\bsgamma} (\bsz)$.

\begin{algorithm}\label{algcbclpspp}
Let $N$, $w_1,\ldots,w_s$, and $Y_1,\ldots,Y_s$ be as above. 
Construct $\bsz=(Y_1 z_1,\ldots,Y_s z_s)$ as follows. 
\begin{itemize}
\item Set $z_1 = 1$.
\item For $d\in [s-1]$ assume that $z_1,\ldots,z_d$ have already been found. Now choose $z_{d+1}\in \cZ_{N,w_{d+1}}$ such that
$$ e_{N,d+1}^2 ((Y_1 z_1,\ldots,Y_d z_{d},Y_{d+1} z_{d+1}))$$
is minimized as a function of $z_{d+1}$. 
\item Increase $d$ and repeat the second step until $(Y_1 z_1,\ldots,Y_s z_s)$ is found. 
\end{itemize}
\end{algorithm}

We show the following theorem which states that our algorithm yields generating vectors $\bsz$ with a 
small integration error. Let $\zeta(x)=\sum_{n=1}^{\infty} n^{-x}$ denote the Riemann zeta function for $x>1$.

\begin{thm}\label{thmcbclpspp}
Let $\bsz=(Y_1 z_1,\ldots,Y_s z_s)\in\ZZ^s$ be constructed according to Algorithm~\ref{algcbclpspp}. 
Then for every $d\in [s]$ it is true that, for $\lambda \in(1/\alpha,1]$, 
\begin{equation}\label{eqthmcbclpspp}
e_{N,d}^2 ((Y_1 z_1,\ldots,Y_d z_d))\le \left(\sum_{\emptyset\neq\uu\subseteq [d]}\gamma_\uu^\lambda 
\frac{2(2\zeta (\alpha\lambda))^{\abs{\uu}}}{b^{\max\{0,m-\max_{j\in \uu} w_j\}}}\right)^{\frac{1}{\lambda}}.
\end{equation}
\end{thm}

The proof of Theorem~\ref{thmcbclpspp} is deferred to the Appendix.

\begin{cor}\label{cor1}
Let $\bsz\in \ZZ^s$ be constructed according to Algorithm~\ref{algcbclpspp}. 
\begin{enumerate}
\item We have $$e_{N,s}(\bsz) \le c_{s,\alpha,\bsgamma,\delta,\bsw} N^{-\alpha/2 +\delta}\ \ \mbox{ for all } 
\delta \in \left(0,\tfrac{\alpha -1}{2}\right],$$ 
where 
$$c_{s,\alpha,\bsgamma,\delta,\bsw}= \left(2 \sum_{\emptyset\neq\uu\subseteq [s]}\gamma_\uu^{\frac{1}{\alpha-2 \delta}}  
\left(2 \zeta\left(\tfrac{\alpha}{\alpha-2 \delta}\right)\right)^{\abs{\uu}} b^{\max_{j\in \uu} w_j}\right)^{\alpha/2-\delta}.$$
\item For $\delta \in \left(0,\tfrac{\alpha -1}{2}\right]$ and $q \ge 0$ define 
$$C_{\delta,q}:=\sup_{s \in \NN}\left[\frac{1}{s^q} 
\sum_{\emptyset\neq\uu\subseteq [s]}\gamma_\uu^{\frac{1}{\alpha-2 \delta}}  
\left(2 \zeta\left(\tfrac{\alpha}{\alpha-2 \delta}\right)\right)^{\abs{\uu}} b^{\max_{j\in \uu} w_j}\right].$$ 
Then we have
\begin{enumerate}
\item If $$C_{\delta,q} < \infty \ \mbox{ for some $\delta \in \left(0,\tfrac{\alpha -1}{2}\right]$ 
and a non-negative $q$},$$ then $$e_{N,s}(\bsz) \le (2 s^q C_{\delta,q})^{\alpha/2-\delta} N^{-\alpha/2 +\delta}$$ 
and hence the worst-case error depends only polynomially on $s$ and $\varepsilon^{-1}$. 
In particular, this implies polynomial tractability with $\varepsilon$-exponent at most $\frac{2}{\alpha-2\delta}$ and an $s$-exponent at most $q$.
\item If $$C_{\delta,0} < \infty \ \mbox{ for some $\delta \in \left(0,\tfrac{\alpha -1}{2}\right]$},$$ 
then $$e_{N,s}(\bsz) \le (2 C_{\delta})^{\alpha/2-\delta} N^{-\alpha/2 +\delta}$$ and hence the worst-case error depends 
only polynomially on $\varepsilon^{-1}$ and is independent of $s$. In particular, this implies strong polynomial tractability with 
$\varepsilon$-exponent at most 
$\frac{2}{\alpha-2\delta}$. 
\item If 
$$\lim_{s \rightarrow \infty} \frac{\log \left(\sum_{\emptyset \not=\uu \subseteq [s]} 
\gamma_{\uu}  (2 \zeta(\alpha))^{|\uu|} b^{\max_{j\in \uu} w_j}\right)}{s}=0,$$ then we obtain weak tractability.
\end{enumerate}
\end{enumerate}
\end{cor}

\begin{proof}
\begin{enumerate}
\item This result follows from Theorem~\ref{thmcbclpspp} by setting $\frac{1}{\lambda}=\alpha-2 \delta$.
\item The proof of these results follows exactly the lines of \cite[Proof of Theorem~16.4]{NW10}.
\end{enumerate}
\end{proof}

Now we consider product weights. Let $\bsgamma_\uu =\prod_{j \in \uu}\gamma_j$ with positive $\gamma_j$ for $j \in \NN$.

\begin{cor}\label{cor2}
Let $\bsz\in \ZZ^s$ be constructed according to Algorithm~\ref{algcbclpspp}. 
\begin{enumerate}
\item The factor $c_{s,\alpha,\bsgamma,\delta,\bsw}$ from Corollary~\ref{cor1} satisfies 
$$c_{s,\alpha,\bsgamma,\delta,\bsw} \le \left(2 \prod_{j=1}^s\left(1+\gamma_j^{\frac{1}{\alpha-2 \delta}}  
2\zeta\left(\tfrac{\alpha}{\alpha-2 \delta}\right) b^{w_j}\right)\right)^{\alpha/2-\delta}.$$
\item If $$A:=\limsup_{s \rightarrow \infty} \frac{\sum_{j=1}^s \gamma_j^{\frac{1}{\alpha-2 \delta}} b^{w_j}}{\log s} < \infty,$$ 
then for all $\eta>0$ there exists a $c_{\eta}>0$ such that 
$c_{s,\alpha,\bsgamma,\delta,\bsw}\le c_{\eta} 2^{\alpha/2-\delta} s^{\zeta(\tfrac{\alpha}{\alpha-2 \delta})(A+\eta)(\alpha-2 \delta)}$ 
and hence the worst-case error depends only polynomially on $s$ and $\varepsilon^{-1}$. 
In particular, this implies polynomial tractability with $\varepsilon$-exponent at most $\frac{2}{\alpha-2\delta}$ 
and an $s$-exponent at most $2 \zeta(\tfrac{\alpha}{\alpha-2 \delta}) A$.
\item If $$B:=\sum_{j=1}^{\infty}\gamma_j^{\frac{1}{\alpha-2 \delta}} b^{w_j} < \infty,$$ 
then $c_{s,\alpha,\bsgamma,\delta,\bsw} \le 2^{\alpha/2-\delta} {\rm e}^{(\alpha-2 \delta) \zeta(\tfrac{\alpha}{\alpha-2 \delta}) B}$ 
and hence the worst-case error depends only polynomially on $\varepsilon^{-1}$ and is independent of $s$. 
This implies strong polynomial tractability with $\varepsilon$-exponent at most $\frac{2}{\alpha-2\delta}$.
\item If $$\lim_{s \rightarrow \infty} \frac{1}{s}\sum_{j=1}^s \gamma_j b^{w_j}=0,$$ then we have weak tractability.
\end{enumerate}
\end{cor}

\begin{proof}
The result follows from Corollary~\ref{cor1} by standard arguments. See also \cite[Proof of Theorem~16.4]{NW10},  
\cite[Proof of Theorem~4]{K03} or \cite[Proof of Theorem~3]{D04}.
\end{proof}

\section{Fast CBC construction for product weights}\label{sec:fast-mod-cbc}

Using the fact that $$\frac{1}{N}\sum_{n=0}^{N-1}\exp(2\pi \icomp k n/N)=\left\{
\begin{array}{ll}
1 & \mbox{ if } k \equiv 0 (N),\\
0 & \mbox{ if } k \not\equiv 0 (N), 
\end{array}\right.$$
the squared worst-case error \eqref{eqerrorexprlps} can be expressed as 
$$e_{N,s,\alpha,\bsgamma}^2 (\bsz)= \sum_{\emptyset\neq\uu\subseteq [s]}\gamma_\uu \sum_{\bsh_\uu \in \ZZ_{\ast}^{\abs{\uu}}}\rho_\alpha (\bsh_\uu) \frac{1}{N} \sum_{n=0}^{N-1} \exp(2 \pi \icomp (\bsh_{\uu} \cdot \bsz_{\uu}) n/N).$$
For product weights $\bsgamma_{\uu}=\prod_{j \in \uu}\gamma_j$ we 
further 
obtain
\begin{equation}\label{eq:fast-cbc-mv1}
e_{N,s,\alpha,\bsgamma}^2 (\bsz)= -1+\frac{1}{N}\sum_{n=0}^{N-1}
\prod_{j=1}^s \left[1+\gamma_j \sum_{h_j \in \ZZ_{\ast}} \frac{\exp(2 \pi
\icomp h_j z_jn/N)}{|h_j|^{\alpha}}\right].
\end{equation}

If $\alpha \ge 2$ is an even integer, then the Bernoulli polynomial $B_{\alpha}$ of degree $\alpha$ has the Fourier expansion $$B_{\alpha}(x)=\frac{(-1)^{(\alpha +2)/2} \alpha!}{(2 \pi)^{\alpha}} \sum_{h \in \ZZ_{\ast}}\frac{\exp(2 \pi \icomp h x)}{|h|^{\alpha}}\;\;\;\mbox{ for all } x \in [0,1).$$ Hence in this case we obtain
\begin{eqnarray*}
e^2_{N,s,\alpha,\bsgamma}(\bsz)=  -1+\frac{1}{N}\sum_{n=0}^{N-1}\prod_{j=1}^s \left[1+\gamma_j \frac{(-1)^{(\alpha +2)/2}(2 \pi)^{\alpha}}{\alpha!} B_{\alpha}\left(\left\{\frac{n z_j}{N}\right\}\right)\right].
\end{eqnarray*}

If $\alpha>1$ is not an even integer we do not have an explicit formula for
the last sum in \eqref{eq:fast-cbc-mv1}. We may nevertheless approximate it
numerically using the (inverse) fast Fourier transform. 

Thus we can express the squared worst-case error as 
\begin{equation}
e^2_{N,s,\alpha,\bsgamma}(\bsz)=  -1+\frac{1}{N}\sum_{n=0}^{N-1}\prod_{j=1}^s \left[1+\gamma_j \varphi_\alpha\left(\left\{\frac{nz_j}{N}\right\}\right)\right] \ ,
\end{equation}
where the values of the function $\varphi_\alpha$ at its $N$ arguments can
be computed at a cost of at most $O(N)$ operations for positive even integers $\alpha$ and stored for later use in the 
CBC algorithm.

We now describe how to perform one step in the CBC algorithm in an efficient 
way. Suppose $z_1,\ldots,z_{d-1}$ have already been computed. If $w_d \ge m$, then we set $z_d = 1$ and no computation is necessary. Thus we assume now that $w_d < m$. Then according to the
reduced fast CBC algorithm we have to find a $z$ which minimizes the error
$e^2_{N,d,\alpha,\bsgamma}(Y_1z_1,\ldots,Y_{d-1}z_{d-1},Y_dz)$ with respect to $z$,
which is equivalent to minimizing
\begin{align*}
\sum_{n=0}^{N-1} \left[1+\gamma_j \varphi_\alpha\left(\left\{\frac{nY_dz}{N}\right\}\right)\right]\eta_{d-1}(n)
= \sum_{n=0}^{N-1} \eta_{d-1}(n)+\gamma_j \sum_{n=0}^{N-1} \varphi_\alpha\left(\left\{\frac{nY_dz}{N}\right\}\right)\eta_{d-1}(n)\,,
\end{align*}
where $\eta_0(n)=1$ and 
\[
\eta_{d-1}(n):=\prod_{j=1}^{d-1} \left[1+\gamma_j\varphi_\alpha\left(\left\{\frac{nY_jz_j}{N}\right\}\right)\right]
\]
for $d\ge 2$. 
Thus we can minimize 
$e^2_{N,d,\alpha,\bsgamma}(Y_1z_1,\ldots,Y_{d-1}z_{d-1},Y_dz)$ with respect to $z$,
by minimizing
\[
T_d(z):=\sum_{n=0}^{N-1} \varphi_\alpha\left(\left\{\frac{nY_dz}{N}\right\}\right)\eta_{d-1}(n)\,.
\]
Now the key observations are that $T_d$ is the product of 
a special $(b^{m-w_d}-1)\times N$-matrix $A=\left(\varphi_\alpha\left(\left\{\frac{nY_dz}{N}\right\}\right)\right)_{z\in \cZ_{N,w_d},n=0,\ldots,N-1}$,  with the vector $\eta_{d-1}$ and that this matrix-vector
product can be computed very efficiently,
as we will show below.

Note that the rows of $A$ are periodic with period
$b^{m-w_d}$, since $Y_d=b^{w_d}$, and $N=b^m$ and therefore $(n+b^{m-w_d}) Y_d z \equiv n Y_d
z(\text{mod }b^m)$. More specifically, $A$ is a block matrix 
\[
A=\Big(\underset{b^{w_d} \text{ - times }}{\underbrace{\bsOmega^{(m-w_d)},\ldots, \bsOmega^{(m-w_d)}}}\Big)\,,
\]
where $\bsOmega^{(k)}:=\left(\varphi_\alpha\left(\left\{\frac{nz}{b^k}\right\}\right)\right)_{z\in \cZ_{b^k,0},n=0,\ldots,b^k-1}$.

If $\bsx$ is any vector of length $N=b^m$ we compute
\[
A \bsx
= \bsOmega^{(m-w_d)}\bsx_1+\cdots+\bsOmega^{(m-w_d)}\bsx_{b^{w_d}}
= \bsOmega^{(m-w_d)}(\bsx_1+\cdots+\bsx_{b^{w_d}})\,,
\]
where $\bsx_1$ consists of the first $b^{(m-w_d)}$ coordinates of $x$,
where $\bsx_2$ consists of the next $b^{(m-w_d)}$ coordinates of $x$,
and so forth.

It has been shown by Nuyens and Cools in \cite{NC06a, NC06b} that multiplication
of a vector of length $b^k$ with $\bsOmega^{(k)}$ can be computed using $O(k b^k)$ operations.
Addition of the vectors $\bsx_1,\ldots,\bsx_{b^{w_d}}$, uses $b^m$ single additions.
Thus multiplication of a $b^m$-vector with $A$ uses $O(b^m+k b^k)$ operations.

We summarize the reduced fast CBC-construction:

\begin{algorithm}[Reduced fast CBC-algorithm]\label{alg:mod-fast-cbc}
Pre-computation:
\begin{enumerate}
\item[a)] Compute $\varphi_\alpha(\frac{n}{b^m})$ for all $n=0,\ldots,b^m-1$.
\item[b)] Set $\eta_1(n)= 1+\gamma_1\varphi_\alpha\left(\left\{\frac{nY_1z_1}{b^m}\right\}\right)$ for $n = 0, \ldots, b^m-1$.
\item[c)] Set $z_1=1$. Set $d=2$ and $s^\ast$ to be the largest integer such that $w_{s^\ast} < m$.
\end{enumerate}

While $d\le \min\{s, s^\ast\}$, 
\begin{enumerate}
\item\label{it:1} partition $\eta_{d-1}$ into $b^{w_d}$ vectors $\eta_{d-1}^{(1)}, \eta_{d-1}^{(2)}, \ldots, \eta_{d-1}^{(b^{w_d})}$ of length $b^{m-w_d}$
and let $\eta' = \eta_{d-1}^{(1)} + \eta_{d-1}^{(2)}+ \cdots + \eta_{d-1}^{(b^{w_d})}$ denote the sum of the parts,
\item\label{it:2} let $T_d(z)=\bsOmega^{(m-w_d)} \eta' $,
\item let $z_d=\mathrm{argmin}_z\ T_d(z)$,
\item\label{it:4} let $\eta_d(n)=\eta_{d-1}(n)\left(1+\gamma_d\varphi_\alpha\left(\left\{\frac{nY_dz_d}{b^m}\right\}\right)\right)$,
\item increase $d$ by $1$.
\end{enumerate}

If $s > s^\ast$, then set $z_{s^\ast} = z_{s^\ast+1} = \cdots = z_s = 0$. The squared worst-case error is then given by 
\[
e^2_{b^m,s,\alpha,\bsgamma}(Y_1z_1,\ldots,Y_{s}z_{s})=-1+\frac{1}{b^m}\sum_{n=0}^{b^m-1}\eta_s(n).
\]
\end{algorithm}
From the above considerations we obtain the following result:
\begin{thm}
For positive even integers $\alpha$, the cost of Algorithm \ref{alg:mod-fast-cbc} is 
$$O \left(b^m+ \min\{s, s^\ast\} b^m+ \sum_{d=1}^{\min\{s, s^\ast\}} (m-w_d)b^{m-w_d} \right).$$
\end{thm}

\begin{proof}
The first term
originates from the pre-computation of $\varphi_\alpha$, the second term
comes from Steps \ref{it:1} and \ref{it:4}, and the last term from Step 
\ref{it:2}.
\end{proof}

\begin{examp}\rm
Assume that the weights satisfy $\gamma_j \sim j^{-3}$. Then a fast CBC algorithm constructs in $O(s m b^m)$ operations a generating vector for which the worst-case error is bounded independently of the dimension $s$. However, if we choose, for example, $w_j \sim \frac{3}{2}\log_b j$, then with the reduced fast CBC algorithm we can construct a generating vector in $O(m b^m + \min\{s, s^\ast\} m)$ operations for which the worst-case error is still bounded independently of the dimension $s$, since $\sum_j \gamma_j b^{w_j} \ll \zeta(3/2) < \infty$. This is a drastic reduction of the construction cost, especially when the dimension $s$ is large.  

We give the results of some practical computations. Throughout we use $b=2$,
$\alpha=2$, $\gamma_j = j^{-3}$ and $w_j =\lfloor \frac{3}{2}\log_b j \rfloor$.
In Table~\ref{tbl:modified-cbc} we report the computation time in seconds 
and the base-10 logarithm of the worst-case error for several values of $s$ and $m$.

\begin{table}[h]
\begin{center}
\begin{tabular}{|c||c|c|c|c|c|c|c|}
\hline
& $s$ = 10& $s$ = 20& $s$ = 50& $s$ = 100& $s$ = 200& $s$ = 500& $s$ = 1000\\
\hline\hline
\multirow{2}{*}{$m$ = 10}&\begin{tabular}{c}0.104\\ -1.89\end{tabular}&\begin{tabular}{c}0.120\\ -1.85\end{tabular}&\begin{tabular}{c}0.144\\ -1.79\end{tabular}&\begin{tabular}{c}0.148\\ -1.74\end{tabular}&\begin{tabular}{c}0.156\\ -1.67\end{tabular}&\begin{tabular}{c}0.164\\ -1.65\end{tabular}&\begin{tabular}{c}0.176\\ -1.65\end{tabular}\\
\hline
\multirow{2}{*}{$m$ = 12}&\begin{tabular}{c}0.356\\ -2.39\end{tabular}&\begin{tabular}{c}0.400\\ -2.35\end{tabular}&\begin{tabular}{c}0.472\\ -2.31\end{tabular}&\begin{tabular}{c}0.524\\ -2.27\end{tabular}&\begin{tabular}{c}0.564\\ -2.19\end{tabular}&\begin{tabular}{c}0.588\\ -2.10\end{tabular}&\begin{tabular}{c}0.608\\ -2.08\end{tabular}\\
\hline
\multirow{2}{*}{$m$ = 14}&\begin{tabular}{c}1.29\\ -2.88\end{tabular}&\begin{tabular}{c}1.45\\ -2.84\end{tabular}&\begin{tabular}{c}1.67\\ -2.79\end{tabular}&\begin{tabular}{c}1.88\\ -2.76\end{tabular}&\begin{tabular}{c}2.03\\ -2.72\end{tabular}&\begin{tabular}{c}2.35\\ -2.62\end{tabular}&\begin{tabular}{c}2.50\\ -2.53\end{tabular}\\
\hline
\multirow{2}{*}{$m$ = 16}&\begin{tabular}{c}5.13\\ -3.39\end{tabular}&\begin{tabular}{c}5.68\\ -3.34\end{tabular}&\begin{tabular}{c}6.47\\ -3.30\end{tabular}&\begin{tabular}{c}7.16\\ -3.28\end{tabular}&\begin{tabular}{c}7.78\\ -3.24\end{tabular}&\begin{tabular}{c}9.27\\ -3.17\end{tabular}&\begin{tabular}{c}11.2\\ -3.10\end{tabular}\\
\hline
\multirow{2}{*}{$m$ = 18}&\begin{tabular}{c}22.3\\ -3.89\end{tabular}&\begin{tabular}{c}24.4\\ -3.84\end{tabular}&\begin{tabular}{c}27.2\\ -3.81\end{tabular}&\begin{tabular}{c}29.4\\ -3.79\end{tabular}&\begin{tabular}{c}32.1\\ -3.76\end{tabular}&\begin{tabular}{c}38.2\\ -3.71\end{tabular}&\begin{tabular}{c}47.2\\ -3.65\end{tabular}\\
\hline
\multirow{2}{*}{$m$ = 20}&\begin{tabular}{c}118\\ -4.41\end{tabular}&\begin{tabular}{c}126\\ -4.35\end{tabular}&\begin{tabular}{c}137\\ -4.33\end{tabular}&\begin{tabular}{c}145\\ -4.31\end{tabular}&\begin{tabular}{c}157\\ -4.30\end{tabular}&\begin{tabular}{c}182\\ -4.26\end{tabular}&\begin{tabular}{c}223\\ -4.21\end{tabular}\\
\hline
\end{tabular}\\
\end{center}
\caption{Computation times and log-worst-case errors for the reduced fast CBC construction}\label{tbl:modified-cbc}
\end{table}

For comparison we also computed the same figures for the fast CBC construction (see Table~\ref{tbl:original-cbc}). The advantage of the modified approach becomes apparent as the dimension
becomes large. This is due to the fact that computation of every component
takes the same amount of time. For example, if $m=20$, computation of one extra
component takes roughly 40 seconds. Thus computing the point set for $m=20$
and $s=1000$ would take roughly 40000 seconds ($\approx$ 11 hours), compared to
223 seconds ($<$ 4 minutes)  for the modified method. Note that the loss
of accuracy, compared to the gain in computation speed, is insignificant.

\begin{table}[h]
\begin{center}
\begin{tabular}{|c||c|c|c|}
\hline
& $s$ = 10& $s$ = 20& $s$ = 50\\
\hline\hline
\multirow{2}{*}{$m$ = 10}&\begin{tabular}{c}0.384\\ -1.90\end{tabular}&\begin{tabular}{c}0.724\\ -1.88\end{tabular}&\begin{tabular}{c}1.80\\ -1.88\end{tabular}\\
\hline
\multirow{2}{*}{$m$ = 12}&\begin{tabular}{c}1.32\\ -2.40\end{tabular}&\begin{tabular}{c}2.62\\ -2.37\end{tabular}&\begin{tabular}{c}6.55\\ -2.37\end{tabular}\\
\hline
\multirow{2}{*}{$m$ = 14}&\begin{tabular}{c}5.22\\ -2.90\end{tabular}&\begin{tabular}{c}10.4\\ -2.87\end{tabular}&\begin{tabular}{c}26.0\\ -2.86\end{tabular}\\
\hline
\multirow{2}{*}{$m$ = 16}&\begin{tabular}{c}21.7\\ -3.40\end{tabular}&\begin{tabular}{c}43.4\\ -3.36\end{tabular}&\begin{tabular}{c}109\\ -3.35\end{tabular}\\
\hline
\end{tabular}\\
\end{center}
\caption{Computation times and log-worst-case errors for the fast CBC construction, i.e., where $w_j=0$ for all $j$}\label{tbl:original-cbc}
\end{table}
\end{examp}

\section{Walsh spaces and polynomial lattice point sets}\label{secnotation}

Similar results to those for lattice point sets from the previous sections can be shown for polynomial lattice 
point sets over finite fields $\FF_b$ of prime order $b$ with modulus $x^m$. Here we only sketch these results and the necessary notations, 
as they are in analogy to those for Korobov spaces and lattice point sets. 

As a quality criterion we use the worst-case error of QMC rules in a weighted Walsh space with 
general weights which was (in the case of product weights) introduced in \cite{DP05} 
(likewise we could also use the mean square worst-case error of digitally shifted polynomial lattices in the Sobolev space  
$\cH_{s,\bsgamma}^{\sob}$ from Remark~\ref{re1}; see \cite{DKPS05,DP05,DP10}). 

For a prime number $b$ and for $h \in \NN$ let $\psi_b(h)=\lfloor \log_b(h)\rfloor$. 
The weighted Walsh space $\cH(K_{s,\alpha,\bsgamma}^{\wal})$ is a reproducing kernel Hilbert space with kernel function of the form 
\begin{align*}
K_{s,\alpha,\bsgamma}^{\wal}(\bsx,\bsy) = 1+ \sum_{\emptyset \not=\uu \subseteq [s]} \gamma_{\uu} \sum_{\bsh_{\uu}\in \NN^{|\uu|}} \frac{\wal_{\bsh_{\uu}}(\bsx_{\uu})\overline{\wal_{\bsh_{\uu}}(\bsy_{\uu})}}{\prod_{j \in \uu}b^{\alpha \psi_b(h_j)}},
\end{align*}
where $\wal_{\bsh}$ denotes the $\bsh$th Walsh function in base $b$ (see, for example, \cite[Appendix~A]{DP10}),
and inner product
$$\langle f,g\rangle_{K_{s,\alpha,\bsgamma}^{\wal}}=\sum_{\uu \subseteq [s]} \gamma_{\uu}^{-1} \sum_{\bsh_{\uu}\in \NN^{|\uu|}} \left(\prod_{j \in \uu}b^{\alpha \psi_b(h_j)}\right) \widetilde{f}((\bsh_{\uu},\bszero)) \overline{\widetilde{g}((\bsh_{\uu},\bszero))},$$ where $\widetilde{f}(\bsh)=\int_{[0,1]^s} f(\bst) \overline{\wal_{\bsh}(\bst)}\rd \bst$ is the $\bsh$th Walsh coefficient of $f$.\\

For integration in $\cH(K_{s,\alpha,\bsgamma}^{\wal})$ we use a special instance of polynomial lattice point sets over $\FF_b$. Let $\cP(\bsg,x^m)$, where $\bsg=(g_1,\ldots,g_s)\in \FF_b[x]^s$, be the $b^m$-element point set consisting of $$\bsx_n:=\left(\nu\left(\frac{n\ g_1}{x^m}\right),\ldots,\nu\left(\frac{n\ g_s}{x^m}\right)\right)\ \ \ \mbox{ for }\ n \in \FF_b[x]\ \mbox{ with } \deg(n)<m,$$ where for $f \in \FF_b[x]$, $f(x)=a_0+a_1 x+\cdots +a_r x^r$, with $\deg(f)=r$ the map $\nu$ is given by $$\nu\left(\frac{f}{x^m}\right):= \frac{a_{\min(r,m-1)}}{b^{m-\min(r,m-1)}}+\cdots+\frac{a_1}{b^{m-1}}+\frac{a_0}{b^m} \in [0,1).$$ 
Note that $\nu(f/x^m)=\nu((f\pmod{x^m})/x^m)$. We refer to \cite[Chapter~10]{DP10} for more information about polynomial lattice point sets. 

The worst-case error of a polynomial lattice rule based on $\cP(\bsg,x^m)$ with $\bsg \in \FF_b[x]^s$ in the weighted Walsh space $\cH(K_{s,\alpha,\bsgamma}^{\wal})$ is given by (see \cite{DKPS05})
\begin{equation}\label{eqerrorexpr}
e_{N,s,\alpha,\bsgamma}^2 (\bsg)=\sum_{\emptyset\neq\uu\subseteq [s]}\gamma_\uu 
\sum_{\bsh_{\uu} \in \D_\uu} \prod_{j \in \uu} b^{-\alpha \psi_b(h_j)},
\end{equation}
where $$\D_\uu:=\left\{\bsh_\uu\in (\FF_b[x]\setminus \{0\})^{\abs{\uu}} \ :\  \bsh_\uu \cdot\bsg_\uu\equiv 0\, (x^m)\right\}.$$

Let again $w_1,\ldots,w_s\in\NN_0$ with $w_1\le w_2\le \cdots \le w_s$ (again the most important case is where $w_1 = 0$, since otherwise each point is counted $b^{w_1}$ times).

We associate a non-negative integer $n$ with $b$-adic expansion $n=n_0+n_1 b+\cdots +n_r b^r$, where $n_j \in \{0,1,\ldots,b-1\}$, with the polynomial $n=n_0+n_1 x+\cdots +n_r x^r$ in $\FF_b[x]$ and vice versa. Despite of this association we point out that in the following we always use polynomial arithmetic in contrast to the previous sections where we used the usual integer arithmetic. 

Now let $Y_j=x^{w_j}\in \FF_b[x]$ for $j=1,2,\ldots,s$. Furthermore, for $w_j<m$, we can write
$$\cZ_{b^m,w_j}:=\left\{h\in\Field_b[x] \setminus\{0\}\ : \  \deg (h) < m-w_j,\ \gcd(h,x^m)=1\right\},$$
and for $w_j\ge m$,
$$ \cZ_{b^m,w_j}:=\{1 \in\Field_b [x]\}.$$
Note that $$|\cZ_{b^m,w_j}|=\left\{
\begin{array}{ll}
b^{m-w_j-1}(b-1) & \mbox{ if } w_j<m,\\
1 & \mbox{ if } w_j \ge m.                                
\end{array}\right.$$

We propose the following CBC construction algorithm for generating vectors $\bsg \in \FF_b[x]^s$. 
\begin{algorithm}\label{algcbcxm}
Let $N,w_1,\ldots,w_s$, $Y_1,\ldots,Y_s$ be as above. Construct $\bsg=(Y_1 g_1,\ldots, Y_s g_s) \in \FF_b[x]^s$ as follows. 
\begin{enumerate}
\item Set $g_1 =1 $.
\item For $d\in [s-1]$ assume that $g_1,\ldots,g_{d}$ have already been found. Now choose $g_{d+1}\in \cZ_{b^m,w_{d+1}}$ such that
$$ e_{N,d+1,\alpha,\bsgamma}^2 ((Y_1 g_1,\ldots,Y_d g_d,Y_{d+1} g_{d+1}))$$
is minimized as a function of $g_{d+1}$. 
\item Increase $d$ and repeat the second step until $\bsg=(Y_1 g_1,\ldots,Y_s g_s)$ is found. 
\end{enumerate}
\end{algorithm}

The following theorem states that our algorithm yields generating vectors $\bsg$ with a small integration error. Let 
\begin{equation*}
\mu_b(\alpha):=\sum_{k=1}^{\infty} b^{-\alpha \psi_b(h)}=\frac{b^{\alpha}(b-1)}{b^{\alpha}-b}.
\end{equation*}

\begin{thm}\label{thmcbcxm}
Let $\bsg=(Y_1 g_1,\ldots,Y_s g_s)\in \FF_b [x]^s$ be constructed according to Algorithm~\ref{algcbcxm}. 
Then for every $d\in [s]$ it is true that, for $\lambda \in(1/\alpha,1]$, 
\begin{equation}\label{eqthmcbcxm}
e_{N,d,\alpha,\bsgamma}^2 ((Y_1 g_1,\ldots,Y_d g_d))\le \left(\frac{b}{b-1}\sum_{\emptyset\neq\uu\subseteq [d]}\gamma_\uu^\lambda 
\frac{\mu_b (\alpha\lambda)^{\abs{\uu}}}{b^{\max\{0,m-\max_{j\in \uu} w_j\}}}\right)^{\frac{1}{\lambda}}.
\end{equation}
\end{thm}

\begin{proof}
The proof of Theorem~\ref{thmcbcxm} is very similar to the proof of Theorem~\ref{thmcbclpspp} and hence we omit it.
\end{proof}

Corollary~\ref{cor1} and Corollary~\ref{cor2} apply accordingly.\\

Finally, the question arises whether a reduced fast CBC construction
analogous to that in Section \ref{sec:fast-mod-cbc} is possible.

For the case of product weights this question can be answered in the 
affirmative, though we will not go into 
details. 

Again we can write the squared worst-case error as
\begin{equation}
e^2_{N,s,\alpha,\bsgamma}(\bsg)=  -1+\frac{1}{N}\sum_{n=0}^{N-1}\prod_{j=1}^s \left[1+\gamma_j \varphi_\alpha\left(\nu\left(\frac{n g_j}{x^m}\right)\right)\right],
\end{equation}
for some function $\varphi_{\alpha}$,
where the numbers $\varphi_\alpha(\nu(\frac{f}{x^m}))$, for $f \in \FF_b[x]$ with $\deg(f)<m$ 
can be computed using $O(N)$ operations, see \cite[Eq. (3.3)]{DKPS05}. 
Minimizing
$e^2_{N,d,\alpha,\bsgamma}(Y_1g_1,\ldots, Y_{d-1} g_{d-1}, Y_d g)$ with respect to $g$ amounts to 
finding the $\mathrm{argmin}_g T_d(g) $, where 
\[
T_d(g)=\sum_{n=0}^{b^m-1}\varphi_\alpha\left(\nu \left(\frac{n\, p_d\, g}{x^m}\right)\right)\eta_{d-1}(n)\,.
\]
Now $T_d$ can be computed efficiently provided multiplication of 
a $b^{k}$ vector with the matrix 
$\Omega^{(k)}=\left(\varphi_\alpha\left(\nu\left(\frac{n \, g}{x^k}\right)\right)\right)_{n,g\in G_{b,k,0}}$ can be performed efficiently for any $k$. 

What one needs for the fast matrix-vector multiplication
is a representation of the group of units of the factor ring $\FF_b[x]/(x^m)$
as a direct product of cyclic groups. 
Indeed, the decomposition can be computed explicitly, 
cf. \cite{SG85}. This explicit decomposition is needed for 
implementing the fast multiplication,
but even without that result we get the following:

\begin{lem}
For any prime number $b$ and any positive integer $k$ let 
\[
\cU_k:=\{g\in \FF_b[x]/(x^k): g \text{ invertible}\}=\{g\in \FF_b[x]/(x^k) \ : \ g(0)\ne 0 \}
\]
denote the group of units of the factor ring $\FF_b[x]/(x^k)$. Then $\cU_k$ can be written as the direct product of at most $k$ cyclic groups.
\end{lem}

\begin{proof}
The fundamental theorem of finite
abelian groups states that $\cU_k$ is a direct product of cyclic groups,
$\cU_k=\bigotimes_{j=1}^r G_j$, say.

$\cU_k$ has $(b-1)b^{k-1}$ elements and the order of any subgroup must 
be a divisor of the group order, i.e., the order of any subgroup must divide 
$(b-1)b^{k-1}$. Now $\cU_1$ is isomorphic to the multiplicative group of 
the finite field $\FF_b$
and is therefore cyclic. Moreover, $\cU_1$ is isomorphic to a subgroup of $\cU_k$. 
Therefore at least one of the cyclic factors of $\cU_k$ contains 
$(b-1)b^i$ elements, where $0\le i \le k-1$. 
All other factors contain at least $b$ elements.

Thus 
\[
(b-1)b^{k-1}=\prod_{j=1}^r |G_j| \ge (b-1)b^{i+r-1}
\ge (b-1)b^{r-1}\,.
\]
So the number $r$ of factors of $\cU_k$ can be at most $k$.
\end{proof}

Now the machinery from \cite{NC06a, NC06b} gives us that multiplication 
with $\Omega^{(k)}$ uses $O(k^2 b^k)$ operations.
Similar to our reasoning in Section \ref{sec:fast-mod-cbc} we conclude
that computation of $T_d$ uses $O((m-w_d)^2 b^{m-e_d})$ operations
and therefore we have:

\begin{thm}
The cost of the reduced fast CBC algorithm is
of order $O(b^m+ \min\{s, s^\ast\} b^m+\sum_{d=1}^{\min\{s, s^\ast\}} (m-w_d)^2 b^{m-w_d})$.
\end{thm}

\section*{Appendix: The proof of Theorem~\ref{thmcbclpspp}}

\begin{proof}
We show the result by induction on $d$. For $d=1$, we have $z_1  = 1$. Thus we have
$$e_{N,1}^2 (Y_1 z_1)=\gamma_{\{1\}}\sum_{\substack{h\in\ZZ_{\ast}\\ h Y_1 z_1\equiv 0\ (N)}} \rho_{\alpha} (h).$$
Let now $\lambda\in(1/\alpha,1]$. 

We now use Jensen's inequality. This inequality states that for non-negative numbers 
$a_k$ and $p\in(0,1]$, it is true that 
$$\sum_k a_k\le \left(\sum_k a_k^p\right)^{1/p}.$$

Applying Jensen's inequality to $e_{N,1}^2 (Y_1 z_1)$, and noting that 
$(\rho_\alpha (h))^\lambda=\rho_{\alpha\lambda} (h)$, we obtain
$$e_{N,1}^{2\lambda} (Y_1 z_1)\le \gamma_{\{1\}}^{\lambda}\sum_{\substack{h\in\ZZ_{\ast}\\ 
h Y_1 z_1\equiv 0\ (N)}} \rho_{\alpha\lambda}(h). $$

If $w_1\ge m$, then $\cZ_{N,w_1}=\{1\}$, $z_1=1$ and $N | b^{w_1}$. In this case, the condition $h Y_1 z_1\equiv 0\ (N)$ is satisfied 
for any $h\in\ZZ_{\ast}$. Consequently,
$$
 e_{N,1}^{2\lambda} (Y_1 z_1)\le \gamma_{\{1\}}^\lambda\sum_{h\in\ZZ_{\ast}} \rho_{\alpha\lambda}(h)\\
= \gamma_{\{1\}}^\lambda 2\zeta(\alpha\lambda)
\le \gamma_{\{1\}}^\lambda \frac{4\zeta(\alpha\lambda)}{b^{\max\{0,m-w_1\}}},
$$
hence \eqref{eqthmcbclpspp} is shown for this case.

If $w_1<m$, then we estimate $e_{N,1}^2 (Y_1 z_1)$ as follows. Since $z_1 = 1$,
\begin{align*}
e_{N,1}^{2\lambda} (Y_1 z_1)\le& \gamma_{\{1\}}^{\lambda} \sum_{\substack{h\in\ZZ_{\ast}\\ h Y_1  \equiv 0\ (N)}} \rho_{\alpha \lambda} (h) =   \gamma_{\{1\}}^\lambda \sum_{\substack{h\in\ZZ_{\ast}\\  b^{m-w_1} | h}} \rho_{\alpha\lambda} (h) = \gamma_{\{1\}}^\lambda \sum_{h \in \ZZ_\ast} \rho_{\alpha \lambda}(h b^{m-w_1}).
\end{align*}
Since $\rho_{\alpha \lambda}(h b^{m-w_1}) = b^{-\alpha \lambda (m-w_1)} \rho(h)$ and $\alpha \lambda > 1$, we obtain
\begin{align*}
e_{N,1}^{2\lambda} (Y_1 z_1) \le & \gamma_{\{1\}}^{\lambda} b^{-\alpha \lambda (m-w_1)} \sum_{h\in\ZZ_{\ast}} \rho_{\alpha \lambda} (h) = \gamma_{\{1\}}^\lambda b^{-\alpha \lambda (m-w_1)} 2 \zeta(\alpha \lambda) \le \gamma_{\{1\}}^\lambda \frac{4\zeta(\alpha\lambda)}{b^{\max\{0,m-w_1\}}},
\end{align*}
hence  \eqref{eqthmcbclpspp} is also shown for this case.

Assume now that we have shown the result for some fixed $d\in [s-1]$, i.e., the generating vector $\bsz_d=(Y_1 z_1,\ldots,Y_d z_d)$ satisfies 
$$e_{N,d}^{2\lambda} ((Y_1 z_1,\ldots,Y_d z_d))\le \sum_{\emptyset\neq\uu\subseteq [d]}\gamma_\uu^\lambda 
\frac{2(2\zeta (\alpha\lambda))^{\abs{\uu}}}{b^{\max\{0,m-\max_{j\in\uu} w_j\}}}.$$
Furthermore, assume that $z_{d+1}\in \cZ_{N,w_{d+1}}$ has been chosen according to Algorithm~\ref{algcbclpspp}. We then have
\begin{align*}
 e_{N,d+1}^2 (\bsz_d,Y_{d+1} z_{d+1})=&\sum_{\emptyset\neq\uu\subseteq [d+1]}\gamma_\uu\sum_{ 
\bsh_\uu\in\D_\uu}\rho_\alpha (\bsh_\uu)\\
=&\sum_{\emptyset\neq\uu\subseteq [d]}\gamma_\uu\sum_{\bsh_\uu\in\D_\uu}
\rho_\alpha (\bsh_\uu)+\sum_{\substack{\emptyset\neq\uu\subseteq [d+1]\\
\{d+1\}\subseteq\uu}}\gamma_\uu\sum_{\bsh_\uu\in\D_\uu}\rho_\alpha (\bsh_\uu)\\ 
=&e_{N,d}^2 (\bsz_d) +\sum_{\substack{\emptyset\neq\uu\subseteq [d+1]\\
\{d+1\}\subseteq\uu}}\gamma_\uu\sum_{\bsh_\uu\in\D_\uu} \rho_\alpha (\bsh_\uu)\\ 
\le&\left(\sum_{\emptyset\neq\uu\subseteq [d]}\gamma_\uu^\lambda 
\frac{2(2\zeta (\alpha\lambda))^{\abs{\uu}}}{b^{\max\{0,m-\max_{j\in\uu} w_j\}}}\right)^{1/\lambda}+
\theta (z_{d+1}),
\end{align*}
where we used the induction assumption and where we write
\begin{equation}\label{eqthetalpspp}
\theta (z_{d+1})=\sum_{\substack{\emptyset\neq\uu\subseteq [d+1]\\
\{d+1\}\subseteq\uu}}\gamma_\uu\sum_{\bsh_\uu\in\D_\uu} \rho_\alpha (\bsh_\uu),
\end{equation}
where the dependence on $z_{d+1}$ in the right hand side is in $\D_\uu$. 

By employing Jensen's inequality, we obtain
\begin{equation}\label{eqJensenlpspp}
e_{N,d+1}^{2\lambda} (\bsz_d,z_{d+1})\le 
 \sum_{\emptyset\neq\uu\subseteq [d]}\gamma_\uu^\lambda 
\frac{2(2\zeta (\alpha\lambda))^{\abs{\uu}}}{b^{\max\{0,m-\max_{j\in\uu} w_j\}}} + \left(\theta (z_{d+1})\right)^{\lambda}.
\end{equation}
We now analyze the expression $\left(\theta (z_{d+1})\right)^{\lambda}$. As $z_{d+1}$ was chosen to minimize the squared worst-case error, 
we obtain
$$\left(\theta (z_{d+1})\right)^{\lambda}\le \frac{1}{|\cZ_{N,w_{d+1}}|}
\sum_{z\in \cZ_{N,w_{d+1}}}\left(\theta(z)\right)^{\lambda},$$
where $\theta(z)$ is the analogue of \eqref{eqthetalpspp} for $z\in \cZ_{N,w_{d+1}}$.
We now have, using Jensen's inequality twice,
\begin{align*}
(\theta(z))^{\lambda}\le&\sum_{\substack{\emptyset\neq\uu\subseteq [d+1]\\
\{d+1\}\subseteq\uu}}\gamma_\uu^\lambda\sum_{\bsh_\uu\in\D_\uu}\rho_{\alpha\lambda} (\bsh_\uu)\\
=&\gamma_{\{d+1\}}^\lambda \sum_{h_{d+1}\in\D_{\{d+1\}}} \rho_{\alpha\lambda}(h_{d+1})\\
&+\sum_{\emptyset\neq\vv\subseteq [d]}\gamma_{\vv\cup\{d+1\}}^\lambda\sum_{h_{d+1}\in\ZZ_{\ast}} \rho_{\alpha\lambda} (h_{d+1}) 
\sum_{\substack{\bsh_\vv\in\ZZ_{\ast}^{\abs{\vv}}\\ \sum_{j\in \vv} h_j Y_j z_j\equiv - h_{d+1} Y_{d+1} z\ (N) }} \rho_{\alpha\lambda} (\bsh_\vv),
\end{align*}
and therefore
\begin{align*}
\lefteqn{ \left(\theta (z_{d+1})\right)^{\lambda}\le\frac{1}{|\cZ_{N,w_{d+1}}|}
\sum_{z\in \cZ_{N,w_{d+1}}}\gamma_{\{d+1\}}^\lambda \sum_{h_{d+1}\in\D_{\{d+1\}}} \rho_{\alpha\lambda}(h_{d+1})}\\
&+\frac{1}{|\cZ_{N,w_{d+1}}|}
\sum_{z\in \cZ_{N,w_{d+1}}}\sum_{\emptyset\neq\vv\subseteq [d]}\gamma_{\vv\cup\{d+1\}}^\lambda\sum_{h_{d+1}\in\ZZ_{\ast}} \rho_{\alpha\lambda} (h_{d+1}) 
\!\!\!\!\!\!\!\!\!\!\sum_{\substack{\bsh_\vv\in\ZZ_{\ast}^{\abs{\vv}}\\ 
\sum_{j\in \vv}h_j Y_j z_j\equiv - h_{d+1} Y_{d+1} z\ (N)}}\!\!\!\!\!\!\!\!\!\! \rho_{\alpha\lambda} (\bsh_\vv)\\
=:&T_1 + T_2.
\end{align*}

For $T_1$, we see, in exactly the same way as for $d=1$, that
\begin{equation}\label{eqT1lpspp}
 T_1\le \frac{4\gamma_{\{d+1\}}^\lambda \zeta (\alpha\lambda)}{b^{\max\{0,m-w_{d+1}\}}}.
\end{equation}

Regarding $T_2$, we again distinguish two cases. If $w_{d+1}\ge m$, we have $\cZ_{N,w_{d+1}}=\{0\}$, and
thus $T_2$ simplifies to
\begin{align}\label{eqT2_1lpspp}
 T_2=&\frac{1}{|\cZ_{N,w_{d+1}}|}
\sum_{z\in \cZ_{N,w_{d+1}}}\sum_{\emptyset\neq\vv\subseteq [d]}\gamma_{\vv\cup\{d+1\}}^\lambda\sum_{h_{d+1}\in\ZZ_{\ast}} \rho_{\alpha\lambda} (h_{d+1}) 
\!\!\!\!\!\!\!\!\!\!\sum_{\substack{\bsh_\vv\in\ZZ_{\ast}^{\abs{\vv}}\\ 
\sum_{j\in \vv} h_j Y_j z_j\equiv 0\ (N)}}\!\!\!\!\!\!\!\!\!\! \rho_{\alpha\lambda} (\bsh_\vv)\nonumber\\
=&\frac{2\zeta (\alpha\lambda)}{b^{\max\{0,m-w_{d+1}\}}}
\sum_{\emptyset\neq\vv\subseteq [d]}\gamma_{\vv\cup\{d+1\}}^\lambda 
\!\!\!\!\!\!\!\!\!\!\sum_{\substack{\bsh_\vv\in\ZZ_{\ast}^{\abs{\vv}}\\ 
\sum_{j\in \vv}h_j Y_j z_j\equiv 0\ (N)}}\!\!\!\!\!\!\!\!\!\! \rho_{\alpha\lambda} (\bsh_\vv)\nonumber\\
<&\frac{4\zeta (\alpha\lambda)}{b^{\max\{0,m-w_{d+1}\}}}
\sum_{\emptyset\neq\vv\subseteq [d]}\gamma_{\vv\cup\{d+1\}}^\lambda 
\sum_{\bsh_\vv\in\ZZ_{\ast}^{\abs{\vv}}} \rho_{\alpha\lambda} (\bsh_\vv)\nonumber\\
=&\frac{4\zeta (\alpha\lambda)}{b^{\max\{0,m-w_{d+1}\}}}
\sum_{\emptyset\neq\vv\subseteq [d]}\gamma_{\vv\cup\{d+1\}}^\lambda 
(2\zeta(\alpha\lambda))^{\abs{\vv}}\nonumber\\
=&\sum_{\emptyset\neq\vv\subseteq [d]}\gamma_{\vv\cup\{d+1\}}^\lambda
\frac{2(2\zeta(\alpha\lambda))^{\abs{\vv}+1}}{b^{\max\{0,m-w_{d+1}\}}}.
\end{align}

On the other hand, if $w_{d+1}<m$, we obtain 
\begin{align*}
T_2=&\frac{1}{b^{m-w_{d+1}-1}(b-1)}\\
& \times \left[
\sum_{z\in \cZ_{N,w_{d+1}}}\sum_{\emptyset\neq\vv\subseteq [d]}\gamma_{\vv\cup\{d+1\}}^\lambda
\sum_{\substack{h_{d+1}\in\ZZ_{\ast}\\ h_{d+1}\equiv 0\ (b^{m-w_{d+1}})}} \rho_{\alpha\lambda} (h_{d+1}) 
\!\!\!\!\!\!\!\!\!\!\sum_{\substack{\bsh_\vv\in \ZZ_{\ast}^{\abs{\vv}}\\ 
\sum_{j\in \vv}h_j Y_j z_j\equiv - h_{d+1} Y_{d+1} z\ (N)}}\!\!\!\!\!\!\!\!\!\! \rho_{\alpha\lambda} (\bsh_\vv)\right.\\
&+ \left.
\sum_{z\in \cZ_{N,w_{d+1}}}\sum_{\emptyset\neq\vv\subseteq [d]}\gamma_{\vv\cup\{d+1\}}^\lambda
\sum_{\substack{h_{d+1}\in\ZZ_{\ast}\\ h_{d+1}\not\equiv 0\ (b^{m-w_{d+1}})}} \rho_{\alpha\lambda} (h_{d+1}) 
\!\!\!\!\!\!\!\!\!\!\sum_{\substack{\bsh_\vv\in \ZZ_{\ast}^{\abs{\vv}}\\ 
\sum_{j\in \vv} h_j Y_j z_j\equiv - h_{d+1} Y_{d+1} z\ (N)}}\!\!\!\!\!\!\!\!\!\! \rho_{\alpha\lambda} (\bsh_\vv)\right]\\
=:&T_{2,1}+T_{2,2}.
\end{align*}
For the term $T_{2,1}$, note that if $h_{d+1}\equiv 0\ (b^{m-w_{d+1}})$, then $h_{d+1} Y_{d+1} z\equiv 0\, (N)$, so we obtain
\begin{align*}
 T_{2,1}=&\frac{1}{b^{m-w_{d+1}-1}(b-1)}\\
& \times \sum_{z\in \cZ_{N,w_{d+1}}}\sum_{\emptyset\neq\vv\subseteq [d]}\gamma_{\vv\cup\{d+1\}}^\lambda
\sum_{\substack{h_{d+1}\in\ZZ_{\ast}\\ h_{d+1}\equiv 0\ (b^{m-w_{d+1}})}}  \rho_{\alpha\lambda} (h_{d+1}) 
\sum_{\substack{\bsh_\vv\in \ZZ_{\ast}^{\abs{\vv}}\\ 
\sum_{j\in \vv} h_j Y_j z_j\equiv 0\ (N)}} \rho_{\alpha\lambda} (\bsh_\vv)\\
=&\sum_{\emptyset\neq\vv\subseteq [d]}\gamma_{\vv\cup\{d+1\}}^\lambda
 \sum_{\substack{\bsh_\vv\in \ZZ_{\ast}^{\abs{\vv}}\\ 
\sum_{j\in \vv}h_j Y_j z_j\equiv 0\ (N)}} \rho_{\alpha\lambda} (\bsh_\vv)
\sum_{\substack{h_{d+1}\in\ZZ_{\ast}\\ h_{d+1}\equiv 0\ (b^{m-w_{d+1}})}} \rho_{\alpha\lambda} (h_{d+1})\\
=&\frac{2\zeta (\alpha\lambda)}{(b^{m-w_{d+1}})^{\alpha\lambda}}\sum_{\emptyset\neq\vv\subseteq [d]}\gamma_{\vv\cup\{d+1\}}^\lambda
 \sum_{\substack{\bsh_\vv\in \ZZ_{\ast}^{\abs{\vv}}\\ 
\sum_{j\in \vv} h_j Y_j z_j\equiv 0\ (N)}} \rho_{\alpha\lambda} (\bsh_\vv).
\end{align*}
From this, it is easy to see that
$$T_{2,1}\le \frac{4\zeta (\alpha\lambda)}{b^{m-w_{d+1}}}\sum_{\emptyset\neq\vv\subseteq [d]}\gamma_{\vv\cup\{d+1\}}^\lambda
 \sum_{\substack{\bsh_\vv\in \ZZ_{\ast}^{\abs{\vv}}\\ 
\sum_{j\in \vv} h_j Y_j z_j\equiv 0\ (N)}} \rho_{\alpha\lambda} (\bsh_\vv).$$

Regarding $T_{2,2}$, note that $h_{d+1}\not\equiv 0\ (b^{m-w_{d+1}})$ and $z\in \cZ_{N,w_{d+1}}$ implies $h_{d+1} Y_{d+1} z \not\equiv 0\ (N)$, 
and for $z_1,z_2\in \cZ_{N,w_{d+1}}$ with $z_1\neq z_2$ we have $h_{d+1} Y_{d+1} z_1 \not\equiv h_{d+1} Y_{d+1} z_2\ (N)$. Therefore,
\begin{align*}
 T_{2,2}\le&\frac{1}{b^{m-w_{d+1}-1}(b-1)}
\sum_{\emptyset\neq\vv\subseteq [d]}\gamma_{\vv\cup\{d+1\}}^\lambda
\sum_{\substack{h_{d+1}\in\ZZ_{\ast}\\ h_{d+1}\not\equiv 0\ (b^{m-w_{d+1}})}} \rho_{\alpha\lambda} (h_{d+1}) 
\!\!\!\!\!\!\!\!\!\!\sum_{\substack{\bsh_\vv\in\ZZ_{\ast}^{\abs{\vv}}\\ 
\sum_{j\in \vv}h_j Y_j z_j\not\equiv 0\ (N)}}\!\!\!\!\!\!\!\!\!\! \rho_{\alpha\lambda} (\bsh_\vv)\\
\le&\frac{2}{b^{m-w_{d+1}}}
\sum_{\emptyset\neq\vv\subseteq [d]}\gamma_{\vv\cup\{d+1\}}^\lambda
\sum_{k_{d+1} \in \ZZ_{\ast}} \rho_{\alpha\lambda} (k_{d+1}) 
\!\!\!\!\!\!\!\!\!\!\sum_{\substack{\bsh_\vv\in\ZZ_{\ast}^{\abs{\vv}}\\ 
\sum_{j\in \vv}h_j Y_j z_j\not\equiv 0\ (N)}}\!\!\!\!\!\!\!\!\!\! \rho_{\alpha\lambda} (\bsh_\vv)\\
=&\frac{4\zeta (\alpha\lambda)}{b^{m-w_{d+1}}}
\sum_{\emptyset\neq\vv\subseteq [d]}\gamma_{\vv\cup\{d+1\}}^\lambda 
\!\!\!\!\!\!\!\!\!\!\sum_{\substack{\bsh_\vv\in\ZZ_{\ast}^{\abs{\vv}}\\ 
\sum_{j\in \vv}h_j Y_j z_j\not\equiv 0\ (N)}}\!\!\!\!\!\!\!\!\!\! \rho_{\alpha\lambda} (\bsh_\vv)\\
=&\frac{4\zeta (\alpha\lambda)}{b^{m-w_{d+1}}}
\sum_{\emptyset\neq\vv\subseteq [d]}\gamma_{\vv\cup\{d+1\}}^\lambda 
\left(\sum_{\bsh_\vv\in\ZZ_{\ast}^{\abs{\vv}}} \rho_{\alpha\lambda} (\bsh_\vv)
-\!\!\!\!\!\!\!\!\!\sum_{\substack{\bsh_\vv\in\ZZ_{\ast}^{\abs{\vv}}\\ 
\sum_{j\in \vv}h_j Y_j z_j\equiv 0\ (N)}}\!\!\!\!\!\!\!\!\!\! \rho_{\alpha\lambda} (\bsh_\vv)\right).
\end{align*}
This yields
\begin{align}\label{eqT2_2lpspp}
T_2=& T_{2,1}+T_{2,2}\nonumber\\ 
\le& \frac{4\zeta (\alpha\lambda)}{b^{m-w_{d+1}}}
\sum_{\emptyset\neq\vv\subseteq [d]}\gamma_{\vv\cup\{d+1\}}^\lambda 
\sum_{\bsh_\vv\in\ZZ_{\ast}^{\abs{\vv}}} \rho_{\alpha\lambda} (\bsh_\vv)\nonumber\\
=&\sum_{\emptyset\neq\vv\subseteq [d]}\gamma_{\vv\cup\{d+1\}}^\lambda 
\frac{2(2\zeta (\alpha\lambda))^{\abs{\vv}+1}}{b^{m-w_{d+1}}}\nonumber\\
=&\sum_{\emptyset\neq\vv\subseteq [d]}\gamma_{\vv\cup\{d+1\}}^\lambda 
\frac{2(2\zeta (\alpha\lambda))^{\abs{\vv}+1}}{b^{\max\{0,m-w_{d+1}\}}}
\end{align}
Combining \eqref{eqT1lpspp}, \eqref{eqT2_1lpspp}, and \eqref{eqT2_2lpspp} yields 
$$(\theta (g_{d+1}))^\lambda\le \sum_{\substack{\uu\subseteq [d+1]\\ \{d+1\}\subseteq\uu}}\gamma_\uu^\lambda 
\frac{2(2\zeta (\alpha\lambda))^{\abs{\uu}}}{b^{\max\{0,m-w_{d+1}\}}}.$$
Plugging into \eqref{eqJensenlpspp}, we obtain
$$e_{N,d+1}^{2 \lambda}(\bsg,g_{d+1}) \le \sum_{\emptyset\neq\uu\subseteq [d]}\gamma_\uu^\lambda 
\frac{2(2\zeta (\alpha\lambda))^{\abs{\uu}}}{b^{\max\{0,m-\max_{j\in\uu} w_j\}}} 
+\sum_{\substack{\uu\subseteq [d+1]\\ \{d+1\}\subseteq\uu}}\gamma_\uu^\lambda 
\frac{2(2\zeta (\alpha\lambda))^{\abs{\uu}}}{b^{\max\{0,m-w_{d+1}\}}}.$$
This yields the result for $d+1$. 
\end{proof}

\noindent {\bf Addresses:} \\

Josef Dick, School of Mathematics and Statistics, The University of New South Wales, Sydney, 2052 NSW, Australia. e-mail: josef.dick(AT)unsw.edu.au \\

Peter Kritzer, Gunther Leobacher, Friedrich Pillichshammer, Department of Financial Mathematics, Johannes Kepler University Linz, Altenbergerstr. 69, 
4040 Linz, Austria. e-mail: {peter.kritzer, gunther.leobacher, friedrich.pillichshammer}(AT)jku.at

\end{document}